\documentclass[11pt]{article}
\usepackage{enumerate}
\usepackage{amsmath}
\usepackage{amsthm}
\usepackage{amsfonts}
\usepackage{amssymb}
\usepackage[numbers]{natbib}
\usepackage{xcolor}
\usepackage{graphicx}
\usepackage{tikz}

\usepackage{fullpage}
\usepackage{xcolor}
\usepackage{hyperref}
\usepackage{bbm}
\newcommand*\samethanks[1][\value{footnote}]{\footnotemark[#1]}

\newtheorem{question}{Question}
\numberwithin{question}{section}
\newtheorem{conjecture}[question]{Conjecture}

\newtheorem{corollary}[question]{Corollary}
\newtheorem{theorem}[question]{Theorem}
\newtheorem{proposition}[question]{Proposition}

\newtheorem{claim}[question]{Claim}

\newtheorem{definition}[question]{Definition}

\newtheorem{remark}[question]{Remark}
\numberwithin{equation}{section}

\title{A threshold for relative hyperbolicity in\\ random 
right-angled Coxeter groups}
\author{Jason Behrstock\thanks{Department of Mathematics, Lehman
College and The Graduate Center, CUNY, New York, USA. Email:
\texttt{jason.behrstock@lehman.cuny.edu}.  Research supported by a
Simons Fellowship.}\and R. Altar
{\c{C}}i{\c{c}}eksiz\thanks{Institutionen f\"or Matematik och
Matematisk Statistik, Ume{\aa} Universitet, Sweden.  Emails:
\texttt{altar.ciceksiz, victor.falgas-ravry@umu.se}.  Research
supported by Swedish Research Council grant VR 2021-03687.} \and
Victor Falgas-Ravry\samethanks}

\begin{document}
\maketitle
\begin{abstract}
We consider the random right-angled Coxeter group $W_{\Gamma}$ whose
presentation graph $\Gamma\sim \mathcal{G}_{n,p}$ is an Erd{\H
o}s--R\'enyi random graph on $n$ vertices with edge probability
$p=p(n)$.  We establish that $p=1/\sqrt{n}$ is a threshold for
relative hyperbolicity of the random group $W_{\Gamma}$.  As a key step in the
proof, we determine the minimal number of pairs of generators that
must commute in a right-angled Coxeter group which is not relatively
hyperbolic, a result which is of independent interest.

We also show that there is an interval of edge probabilities of width
$\Omega(1/\sqrt{n})$ in which the random
right-angled Coxeter group has precisely cubic divergence. This 
interval is between the thresholds for relative 
hyperbolicity (whence exponential divergence) 
and quadratic divergence. Moreover, a simple random walk on any Cayley graph 
of the random right-angled Coxeter group for $p$ in this interval satisfies a central limit theorem.
\end{abstract}

\section{Introduction}
 
In his famous treatise~\cite{Gromov:hyperbolic}, Gromov initiated the 
study of random groups with the statement that ``almost all'' groups 
are hyperbolic. This idea was developed further  in work by Gromov and 
others, who established that for a number of random group models as  one varies a model parameter one sees a threshold at which the typical behaviour changes from one regime into another: on one side of the threshold, random groups outputted by the model are typically infinite and hyperbolic, while on the other side of the threshold they are typically finite, see e.g.\  Gromov~\cite{Gromov:asymptotic}, Ollivier~\cite{Ollivier:hyperrandomgroups} and Ol'shanskii~\cite{Olshanskii:hyperbolic}. The world of these random group models thus splits into two distinct regimes. As we shall see, this stands in some contrast with the model studied in this paper.

The right-angled Coxeter group (or RACG) $W_{\Gamma}$ with
presentation graph $\Gamma=(V,E)$ is the group with generators $V$ and
relations $a^2=\mathrm{id}$ and $ab=ba$ for all $a\in V$ and $ab\in
E$.  We will use the Erd{\H o}s--R\'enyi random graph model, which is
defined as follows: for $n\in \mathbb{N}$ and a sequence of
probabilities $p=p(n)\in [0,1]$, we define a random graph
$\mathcal{G}_{n,p}$ on the vertex set $[n]:=\{1,2, \ldots n\}$ by
including each of the $\binom{n}{2}$ possible edges with probability
$p(n)$, independently at random.  We write $\Gamma\sim
\mathcal{G}_{n,p}$ to denote the fact that $\Gamma$ is a random graph
with the same distribution as $\mathcal{G}_{n,p}$.  To any $\Gamma\sim
\mathcal{G}_{n,p}$ one can associate the RACG $W_{\Gamma}$, thereby
obtaining a model for a \emph{random right-angled Coxeter group}.

The random right-angled Coxeter group was shown by Charney--Farber to be typically infinite and hyperbolic if 
$pn\to 0$, and finite if $(1-p)n^{2}\to 0$ 
\cite[Corollary~2.2]{CharneyFarber}. However, unlike in the types of models 
introduced by Gromov, there is a universe of interesting behavior which  
occurs between these two extremes. In this paper we make progress towards understanding the algebraic and 
geometric properties of such groups via combinatorial and probabilistic tools.

The past two decades have revealed that generalizations of
hyperbolicity provide a powerful framework for studying finitely
generated groups.  One of the strongest types of such ``non-positive
curvature'' is called \emph{relative
hyperbolicity}. This notion was introduced by
Gromov~\cite{Gromov:hyperbolic}, then developed by
Farb~\cite{Farb:RelHyp} and eventually given many equivalent
geometric, topological, and dynamical formulations, see 
Bowditch~\cite{Bowditch:RelHyp}, Dahmani~\cite{Dahmani:thesis},
Dru{\c{t}}u and Sapir~\cite{DrutuSapir:TreeGraded},
Osin~\cite{Osin:RelHyp}, Sisto~\cite{Sisto:metricrelhyp,
Sisto:projrelhyp}, Yaman~\cite{Yaman:RelHyp}, and others.  Relative 
hyperbolicity is both
general enough to include many important classes of groups including
for example all fundamental groups of finite-volume hyperbolic
manifolds and all finitely generated groups with infinitely many ends,
while being sufficiently restrictive to yield powerful implications
concerning geometric, algebraic, and algorithmic properties, see 
Arzhantseva, Minasyan and
Osin~\cite{ArzhantsevaMinasyanOsin:SQ-universality},
Dru{\c{t}}u~\cite{Drutu:RelHyp}, Dru{\c{t}}u and
Sapir~\cite{DrutuSapir:Splitting}, Farb~\cite{Farb:RelHyp}, and many others.

Roughly speaking, a group is \emph{relatively hyperbolic} when all its
`non-hyperbolicity' is confined to certain subgroups that do not
interact with each other in any substantial way.
In this paper, rather than working directly with the definition of
relative hyperbolicity, we will rely on two of the main results of Behrstock,
Hagen and Sisto~\cite[Theorems I and II]{BehrstockHagenSisto:coxeter},
which, in the context of RACGs, 
establishes a necessary and sufficient criterion for relative
hyperbolicity in terms of combinatorial properties of the 
presentation graph.

The main result in this paper is the following, which shows
that a random RACG transitions from being asymptotically almost surely
(a.a.s.) relatively hyperbolic to being a.a.s.\ not relatively
hyperbolic around $p=1/\sqrt{n}$.

\begin{theorem}\label{theorem: main}
Let $p=p(n)$ and $\Gamma\sim \mathcal{G}_{n,p}$. Then the following hold:
\begin{enumerate}[(i)]
	\item if $p\leq \frac{1}{4\sqrt{n\log (n) }}$, then $W_{\Gamma}$ is a.a.s.\ relatively hyperbolic;
	\item if $p\geq \frac{\sqrt{\sqrt{6}-2}}{\sqrt{n}}$, then $W_{\Gamma}$ is a.a.s.\ not relatively hyperbolic.
\end{enumerate}
\end{theorem}

We note that this threshold 
provides a wide range of RACGs which are both relatively hyperbolic and  
one-ended, since the threshold for one-endedness occurs much lower, at 
 $p=\frac{\log(n)}{n}$, see~\cite[Theorem~3.2]{BehrstockHagenSisto:coxeter}.

A powerful geometric invariant for distinguishing finitely generated
groups is their \emph{divergence}.  This notion was introduced by
Gromov~\cite{Gromov:asymptotic} and refined by
Gersten~\cite{Gersten:divergence,Gersten:divergence3mflds} and,
roughly speaking, is a function of $r$ providing a measure of the
length of a shortest path needed to connect a pair of points at
distance $r$ in a geodesic space while avoiding a ball of radius 
linear in $r$ 
around a third point.  In Euclidean space the divergence
is linear in $r$, while in hyperbolic and relatively hyperbolic spaces
it is exponential in $r$.

Dani and Thomas~\cite{DaniThomas:divcox} gave a construction of RACGs with 
each possible degree of polynomial 
divergence. On the other hand, combining work of Behrstock--Hagen--Sisto~\cite{BehrstockHagenSisto:coxeter} and
Levcovitz~\cite{Levcovitz:DivergenceRACG} it follows that 
whenever a RACG is not relatively hyperbolic, then its divergence must be 
polynomial.

The degree of polynomial divergence
of the Cayley graph of a right-angled Coxeter group can be computed via a geometric
invariant called \emph{thickness}\footnote{In the cases of 
linear and quadratic divergence the characterization is in terms of 
\emph{strong algebraic thickness}, for higher orders this notion 
does not characterize (see \cite{Levcovitz:QIinvariantThickness} for
an obstruction), but a more flexible invariant called \emph{strong thickness}
does characterize.}, which in turn can be
characterized in terms of the combinatorics of the presentation graph
by results of Behrstock, Hagen and
Sisto~\cite{BehrstockHagenSisto:coxeter} and
Levcovitz~\cite{Levcovitz:QIinvariantThickness}.  This combinatorial
characterization results in a notion of \emph{graph theoretic
thickness}, which is how we shall define and apply thickness in this
paper (see Section~\ref{section: def thickness} for the 
definition).

We investigate the order of thickness and divergence in the random
RACG $W_{\Gamma}$, $\Gamma\sim \mathcal{G}_{n,p}$, in the $p$--regime
when $W_{\Gamma}$ is a.a.s.\ not relatively hyperbolic.  The problem
of determining the threshold for thickness and divergence of various
orders was raised by Behrstock, Hagen and Sisto
in~\cite[Question 1]{BehrstockHagenSisto:coxeter}.  In prior work,
Behrstock, Falgas--Ravry and Susse~\cite{behrstock2022square}
established a sharp threshold for thickness of order 1 and quadratic
divergence.

Gromov stated in \cite[$\S 6.B_2.h$]{Gromov:asymptotic} that he 
``expected'' that in non-positively curved spaces the divergence is 
either linear or exponential. In contrast to this, Gersten responded 
by showing  
that the divergence of $3$--manifold groups is either linear, 
quadratic, or exponential  
\cite{Gersten:divergence3mflds}. In his paper, Gersten then highlighted the question of 
whether or not it was possible for a CAT(0) cube complex admitting a 
geometric group action to 
have polynomial divergence of degree $3$ or higher. Several groups 
worked on this, but it took more than a 
decade before examples were constructed ---  this was done 
simultaneously 
in two independent papers of Behrstock and Dru\c{t}u~\cite{BehrstockDrutu:thick2} and Macura~\cite{Macura:polydiv}. The following theorem, in 
sharp contrast to Gromov's expectation, 
establishes that there is a non-trivial 
range in a random group model where the random group is a 
naturally occurring CAT(0) group with cubic divergence. The proof,  
adapting the ideas and arguments of~\cite{behrstock2022square}, 
shows that thickness of order $2$ and cubic divergence in random RACG
occur a.a.s.\ for $p=p(n)$ in some interval of values of width
$\Omega\left(1/\sqrt{n}\right)$ below the sharp threshold for
thickness of order $1$ and quadratic divergence.

\begin{theorem}\label{theorem: threshold for thickness of order 2 below that for order 1}
	There exists an absolute constant $c>0$ such that the following hold:
	\begin{enumerate}[(i)]
		\item 
		for $p=p(n)$ satisfying $\frac{\sqrt{\sqrt{6}-2}-c}{\sqrt{n}}\leq p(n)\leq 1- \Omega(\frac{\log n}{n})$, a.a.s.\ the random graph $\Gamma\sim\mathcal{G}_{n,p}$ is thick of order at most $2$;
		\item for every fixed $\varepsilon>0$ and $p=p(n)\in [\frac{\sqrt{\sqrt{6}-2}-c}{\sqrt{n}},\frac{\sqrt{\sqrt{6}-2}-\varepsilon}{\sqrt{n}}]$, a.a.s.\ the random graph $\Gamma\sim\mathcal{G}_{n,p}$ is thick of order exactly $2$, and the random group $W_{\Gamma}$ has cubic divergence.
	\end{enumerate} 
\end{theorem}
\noindent As an immediate consequence of Theorem~\ref{theorem: threshold for
thickness of order 2 below that for order 1} and a result of Chawla,
Choi, He and Rafi~\cite[Proposition 1.1]{ChawlaChoiHeRafi:randomdiv} 
(a result about RACGs which are thick of order at least 2)
we have the following:
\begin{corollary}\label{cor: random walks}	
	There exists an absolute constant $c>0$ such that for any fixed $\varepsilon$ with $0<\varepsilon<c$ and $p=p(n)\in
	[\frac{\sqrt{\sqrt{6}-2}-c}{\sqrt{n}},\frac{\sqrt{\sqrt{6}-2}-\varepsilon}{\sqrt{n}}]$,
	a.a.s.\ the random graph $\Gamma\sim\mathcal{G}_{n,p}$ has the 
	property that the simple random walk on any Cayley graph of the right-angled Coxeter group $W_{\Gamma}$ satisfies a central limit theorem and has a normal limit law.
\end{corollary}

A key ingredient in the proof of Theorem~\ref{theorem: main} is the following extremal result on
the minimum number of edges required for a graph to be thick (meaning
thick of some finite order), which is of independent interest.
\begin{theorem}\label{theorem: combinatorial characterisation of thickness}
	Let $\Gamma$ be an $m$--vertex graph that is thick.  Then
	$\vert E(\Gamma)\vert \geq 2m-4$.
\end{theorem}
\noindent The lower bound in Theorem~\ref{theorem: combinatorial characterisation of thickness} is best possible, since, e.g., for all $m\geq 4$ the complete bipartite graph $K_{2, m-2}$ is
thick of order $0$ and has exactly $2m-4$ edges.  There are however 
many other extremal examples: consider the graph on $\{1,2, \ldots, 
m\}$ obtained by joining $i,j$ by an edge if and only if $\vert 
\lceil i/2\rceil - \lceil j/2\rceil\vert =1$ (this graph can be 
visualized as a `path of squares', see Figure~\ref{fig 1} for an 
example). Again, it is easy to verify that this graph is thick of 
order $1$ and has exactly $2m-4$ edges.  Further, one can identify 
any pair of non-edges present in two extremal configurations on $m_1$ 
and $m_2$ vertices to obtain a new extremal configuration on $m_1+m_2-2$ vertices --- see Figure~\ref{fig 2} for an example. In this way one can obtain a wide variety of extremal examples, presenting a particular challenge in the proof of Theorem~\ref{theorem: combinatorial characterisation of thickness}. A somewhat surprising feature of our result is that the order of thickness of $\Gamma$ does not change the minimum number of edges required: in general, the price one has to pay for thickness of order $0$ is no higher than the price one has to pay for thickness of higher orders.

Theorem~\ref{theorem: combinatorial characterisation of thickness},
taken together with a result of Behrstock, Hagen and
Sisto~\cite[Theorem I]{BehrstockHagenSisto:coxeter} showing that a
RACG is relatively hyperbolic if and only it is not thick, immediately
implies the following tight lower bound on the minimum number of
commutative relations needed in a RACG to prohibit relative
hyperbolicity.

\begin{corollary}\label{cor: min number relations of not rel 
	hyperbolicity} Consider a graph $\Gamma$ which we take to be the 
	presentation graph of a right-angled Coxeter
group $W_{\Gamma}$. If $\vert E(\Gamma)\vert < 2\vert V(\Gamma)\vert-4,$ then $W_{\Gamma}$ is relatively
hyperbolic.
\end{corollary}

\subsection*{Questions and conjectures}
We believe the following sharpening of Theorem~\ref{theorem: main}
might hold. We expect that to prove such a result will require either an entirely new
approach or a substantial quantitative strengthening of our
techniques:
\begin{conjecture}\label{conj: sharp threshold for relative hyperbolicity}
	Let $\Gamma\sim \mathcal{G}_{n,p}$. Then the following hold:
	\begin{enumerate}[(i)]
		\item if $p=p(n)=o\left(\frac{1}{\sqrt{n}}\right)$, then a.a.s.\ the
		right-angled Coxeter group $W_{\Gamma}$ is relatively hyperbolic; 
		\item if
		$p=p(n)=\Omega\left(\frac{1}{\sqrt{n}}\right)$, then a.a.s.\ the
		right-angled Coxeter group $W_{\Gamma}$ is thick of order $O(1)$ and
		has polynomial divergence.
	\end{enumerate}
\end{conjecture}
\noindent A key difficulty presented by this conjecture is the question of whether or not a random graph $\Gamma\sim \mathcal{G}_{n,p}$ with $p=p(n)=o\left(\frac{1}{\sqrt{n}}\right)$ may have thickness of order $t=t(n)$ finite but tending to infinity as $n\rightarrow \infty$.

In an even more challenging direction, we believe that Theorem~\ref{theorem: threshold for thickness of order 2 below that for order 1} can be generalized, and that for any fixed $k\in \mathbb{Z}_{\geq 3}$ there exist intervals of width $\Omega\left(1/\sqrt{n}\right)$ of values of $p(n)$ for which $W_{\Gamma}$ a.a.s.\ exhibits degree $k$ polynomial divergence. Accordingly, we make the following bold conjecture, strengthening Conjecture~\ref{conj: sharp threshold for relative hyperbolicity} part (ii):
\begin{conjecture}\label{conj: thickness thresholds}
	There exists an infinite strictly decreasing sequence of strictly positive real numbers $(\lambda_k)_{k\in \mathbb{N}}$
	\[\lambda_1=\sqrt{\sqrt{6}-2}> \lambda_2>\lambda_3>\cdots >\lambda_k> \ldots \, ,\] such that for every $k \in \mathbb{N}$ and every $\varepsilon>0$ fixed, the following hold:
	\begin{enumerate}[(i)]
		\item	for $p=p(n)\leq  \frac{\lambda_k-\varepsilon}{\sqrt{n}}$, a.a.s.\ the random graph $\Gamma\sim\mathcal{G}_{n,p}$ is not thick of order at most $k$;
		\item 	for $ p=p(n)$ satisfying $\frac{\lambda_k+\varepsilon}{\sqrt{n}}\leq p(n) \leq 1- \Omega(\frac{\log n}{n})$, a.a.s.\ the random graph $\Gamma\sim\mathcal{G}_{n,p}$ is thick of order at most $k$.
	\end{enumerate}	
\end{conjecture}
Showing the existence of such constants $\lambda_k$ (let alone
determining their values!)  seems however an extremely difficult problem even in
the case $k=2$ (although Theorem~\ref{theorem: threshold for thickness of 
order 2 below that for order 1} is evidence in this direction).  Indeed, it is not even clear what the right rate of
decay for the sequence $\lambda_k$ should be: does it tend to $0$, and
if so at what speed? As discussed in Remark~\ref{remark: 
nonmonotonicity} the property of being thick of 
order~$k$ is not monotonic, which is another difficulty in 
approaching this conjecture.

Finally, as already discussed by Behrstock, Falgas--Ravry and Susse
in~\cite{behrstock2022square}, it could be fruitful to investigate
similar questions to the ones considered in this paper when the presentation
graph $\Gamma$ is taken from a different random graph distribution
than that given by the Erd{\H o}s--R\'enyi model.  Random $d$-regular
graphs or random graph models with clustering such as random
intersection graphs or random geometric graphs could be interesting
models to study in this way, as the typical geometric properties of
the resulting RACG may differ in novel ways from those established in this paper.

\subsection*{Related results}
The study of random right-angled Coxeter groups and their geometric
and cohomological properties was initiated in papers of Charney and
Farber~\cite{CharneyFarber}, Davis and Kahle~\cite{DavisKahle} and
Behrstock, Hagen and Sisto~\cite{BehrstockHagenSisto:coxeter} amongst
others; these results followed Costa and Farber's introduction of a 
related model of random right-angled Artin groups~\cite{CostaFarber}.  Theorem~\ref{theorem: main} part (i) of this paper represents
a dramatic improvement on~\cite[Theorem
III]{BehrstockHagenSisto:coxeter}, where it was established that for
$p(n)=o(n^{-5/6})$, a.a.s.\ the random RACG $W_{\Gamma}$, $\Gamma\sim
\mathcal{G}_{n,p}$, is relatively hyperbolic: this represents significant
progress towards the completion of a systematic picture of the
geometric properties of random RACGs proposed in~\cite[Figure
4]{BehrstockHagenSisto:coxeter}.

Improving on the earlier work of~\cite{BehrstockHagenSisto:coxeter},
Behrstock, Falgas--Ravry, Hagen and Susse~\cite{behrstock2018global}
established a threshold result for thickness of order $1$ and
quadratic divergence of the random RACG $W_{\Gamma}$,
$\Gamma\sim\mathcal{G}_{n,p}$ around $p=1/\sqrt{n}$.  More precisely,
they showed that for $p=p(n)\leq 1/\left(\log(n)\sqrt{n}\right)$,
$W_{\Gamma}$ exhibits a.a.s.\ at least cubic divergence, while for
$p=p(n)\geq 5\sqrt{\log n}/\sqrt{n}$ it has a.a.s.\ at most quadratic
divergence.  This result was later sharpened by Behrstock,
Falgas--Ravry and Susse, who determined
in~\cite[Theorem~1.6]{behrstock2022square} the precise location of the
transition between cubic and quadratic divergence with the following 
result.
	\begin{theorem}[Behrstock, Falgas--Ravry and Susse]\label{theorem: threshold order 1 thickness}
		Let $c>0$ be fixed, and let $\Gamma\sim \mathcal{G}_{n,p}$. The following hold:
		\begin{enumerate}[(i)]
			\item if $c<\sqrt{\sqrt{6}-2}$ and $p=p(n)\leq c/\sqrt{n}$, then a.a.s.\ $W_{\Gamma}$ is not algebraically thick of order $1$ and has at least cubic divergence;
			\item if $c>\sqrt{\sqrt{6}-2}$ and $c/\sqrt{n}\leq p(n)$, then a.a.s.\ $W_{\Gamma}$ is algebraically thick of order at most $1$ and has at most quadratic divergence. Moreover if there is some constant $\varepsilon>0$ such that $p(n)\leq 1-(1+\varepsilon)\frac{\log n}{n}$ then in fact $W_{\Gamma}$ is algebraically thick of order exactly $1$ and has quadratic divergence.
		\end{enumerate}
	\end{theorem}

%
%

More recently, Susse~\cite{susse2023morse} studied Morse subgroups and
Morse boundaries of random RACGs. In separate
work~\cite{BehrstockCiceksizFalgasRavry:connectivity}, the authors of
the present paper establish a threshold for connectivity of the square
graph of a random graph, a result that settles a conjecture of Susse
and has geometric applications to the study of cubical coarse rigidity
in random RACG.

Finally, it would be remiss of us to close this section without mentioning the closely related topic of \emph{clique percolation}, which arises in the study of random graphs.  Clique percolation first appeared as a simple model for community detection in a work of Der\'enyi, Palla and Vicsek~\cite{DerenyiPallaVicsek:CliquePercolation}, and was then extensively featured in the network science literature~\cite{LiDengWang:cliqueperc, DerenyiPallaVicsek:CommunityStructure,
	TothVicsekPalla:overlappingmodularitycliqueperc,WangCaoSuzukiAihara:Epidemicscliqueperc}.
	In $(k,\ell)$--clique percolation, given a graph $\Gamma$ one 
	forms an auxiliary graph $K_{k,\ell}(\Gamma)$ whose vertices are 
	the $k$--cliques of $\Gamma$ and whose edges are those pairs of $k$--cliques having $\ell$ vertices in common.  One of the main research questions in the area was to determine the threshold $p$ for the emergence of a giant component in $K_{k, \ell}(\Gamma)$ when $\Gamma\sim \mathcal{G}_{n,p}$ is an Erd{\H o}s--R\'enyi random graph.  This was achieved in a landmark 2009 paper of Bollob\'as and Riordan~\cite{BollobasRiordan:cliquepercolation}, using highly sophisticated branching process techniques that in turn underpinned the arguments deployed by Behrstock, Falgas--Ravry and Susse to study square percolation and determine the threshold for quadratic divergence in random RACG in~\cite{behrstock2022square}.

\subsection{Organization of the paper}
In Section~\ref{section: preliminaries} we summarize the graph
theoretic and probabilistic notions we shall use, as well
as discuss divergence and a combinatorial description of
thickness that will play a key role in the proof of
Theorem~\ref{theorem: combinatorial characterisation of thickness}.
We then prove our extremal results Theorem~\ref{theorem: combinatorial
characterisation of thickness} and Corollary~\ref{cor: min number
relations of not rel hyperbolicity} in Section~\ref{section: extremal
result}, while Section~\ref{section: thickness} is devoted to 
proofs of Theorems~\ref{theorem: main} and
~\ref{theorem:
threshold for thickness of order 2 below that for order 1}.

\section{Preliminaries}\label{section: preliminaries}

\subsection{Graph theoretic notions and notation}
We recall here some standard graph theoretic notions and notation that we will use throughout the paper. We write $[n]:=\{1,2, \ldots, n\}$, and $x_1x_2\ldots x_r$ to denote the $r$-set $\{x_1, x_2, \ldots, x_r\}$. Given a set $S$, we let $S^{(r)}$ denote the collection of all subsets of $S$ of size $r$.

A graph is a pair $\Gamma=(V,E)$, where $V=V(\Gamma)$ is a set of vertices and $E=E(\Gamma)$ is a subset of $V^{(2)}$. All graphs considered in this paper are thus simple graphs, with no loops or multiple edges. We use $v(\Gamma):=\vert V(\Gamma)\vert$  and $e(\Gamma):=\vert E(\Gamma)\vert$ to denote respectively the order and the size (i.e.\ the number of vertices and the number of edges) of a graph.  Two subgraphs $\Gamma$ and $\Gamma'$ are isomorphic if there exists a bijection from $V(\Gamma)$ to $V(\Gamma')$ taking edges to edges and non-edges to non-edges; we denote this fact by $\Gamma \cong \Gamma'$.

A subgraph of $\Gamma$ is a graph $\Gamma'$ with $V(\Gamma')\subseteq
V(\Gamma)$ and $E(\Gamma')\subseteq E(\Gamma)$.  We say $\Gamma'$ is
the subgraph of $\Gamma$ induced by a set of vertices $S$ if $\Gamma'
=(S, S^{(2)}\cap E(\Gamma))$, and denote this fact by writing
$\Gamma'=\Gamma[S]$.  The complement of $\Gamma$ is the graph
$\overline{\Gamma}:=(V, V^{(2)}\setminus E)$.  
Given a vertex $x$ in $\Gamma$, we denote the set of its neighbors by $N_{\Gamma}(x):=\{y\in V(\Gamma): \ xy \in
E(\Gamma)\}$.

A path of length $\ell\geq 0$ in $\Gamma$ is a sequence of distinct vertices $v_0, v_1, \ldots, v_{\ell}$ with $v_iv_{i+1}\in E(\Gamma)$ for all $i\in [\ell-1]\cup\{0\}$. The vertices $v_0$ and $v_{\ell}$ are called the endpoints of the path. Two vertices are said to be connected in $\Gamma$ if they are the endpoints of some path of finite length. Being connected in $\Gamma$ is an equivalence relation on the vertices of $\Gamma$, whose equivalence classes form the connected components of $\Gamma$.  If there is a unique connected component, then the graph $\Gamma$ is said to be connected. A minimally connected subgraph of a connected graph is called a
spanning tree. Two useful facts about trees we shall use in our argument are (i) that every connected graph contains a spanning tree, and (ii) that the vertex set of a tree can be ordered starting from any vertex $v$ as $v=v_1$, $v_2$, $\ldots$, $v_n$ in such a way that for every $j>1$ the vertex $v_j$ sends at least one edge into the set $\{v_i: \ 1\leq i <j\}$.





Finally, we denote by $K_{1,2}$ the \emph{cherry}, or induced path on
three vertices, $K_{1,2}:=([3], \{12, 23\})$; we let $C_4$ denote the
\emph{square} (or \emph{$4$-cycle}), $C_4:=([4], \{12,23,34, 14\})$;
and we write $K_m$ for the \emph{clique} (or \emph{complete graph}) of
order $m$, $K_m:=([m], [m]^{(2)})$.

\subsection{Probabilistic notation and tools}
We write $\mathbb{P}$ and $\mathbb{E}$ and $\mathrm{Var}$ for probability, expectation and variance respectively. 
We say that a sequence of events $\mathcal{E}=\mathcal{E}(n)$, $n\in \mathbb{N}$  holds \emph{a.a.s.}\ (asymptotically almost surely) if $\lim_{n\rightarrow \infty} \mathbb{P}(\mathcal{E}(n))=1$. Throughout the paper we use standard Landau notation: for functions $f,g: \ \mathbb{N}\rightarrow \mathbb{R}$ we write $f=o(g)$ if $\lim_{n\rightarrow \infty} f(n)/g(n)=0$, $f=O(G)$ if there exists a real constant $C>0$ such that  $\limsup_{n\rightarrow \infty} \vert f(n)/g(n)\vert \leq C$. We further write $f=\omega(g)$ if $g=o(f)$, and $f=\Omega(g)$ if $g=O(f)$.

We shall make repeated use of Markov's inequality: given a non-negative integer-valued random variable $X$,  $\mathbb{P}(X>a)\leq \frac{1}{a+1}\mathbb{E}X$ for any integer $a \geq 0$.

\subsection{Divergence}
A useful object in geometry is the \emph{divergence function} of a
geodesic space.  Roughly speaking, this function measures the length
of a shortest path between pairs of points at distance $r$ apart in
the space when forced to avoid a ball of radius $r$ around a third
point.  There are a number of different ways to formalize this and it
is proven in~\cite[Proposition~3.12]{DrutuMozesSapir} that the various
definitions in the literature carry the same information.  For more
background on divergence, we refer the reader to, e.g.,  Behrstock and
Dru{\c{t}u}~\cite{BehrstockDrutu:thick2} or Dru{\c{t}}u, Mozes and
Sapir~\cite{DrutuMozesSapir}.

A quasi-isometry preserves the divergence function up to affine 
functions. Hence the \emph{growth rate of the divergence function} is a 
quasi-isometry invariant and thus is the same for all Cayley graphs of 
a given finitely generated group, 
independently of the choice of the finite generating set. Accordingly, we will abuse language 
slightly and just say the divergence function is linear, quadratic, cubic,
polynomial, etc, when strictly speaking we are talking about the 
growth rate of the divergence function.  

\subsection{Thickness and the square graph}\label{section: def thickness}

An important role in this paper is played by a graph theoretic notion
of thickness.  \emph{Graph-theoretic thickness} was introduced by
Behrstock, Hagen and Sisto in~\cite{BehrstockHagenSisto:coxeter} as a
combinatorial way of capturing a geometric/algebraic property of groups called
\emph{algebraic thickness}.

This latter notion of thickness first appeared in the context of
geometric group theory in work of Behrstock, Dru\c{t}u and
Mosher~\cite{BDM}.  Algebraic thickness is defined inductively.
Roughly speaking, the base level of this property (algebraic thickness
of order zero) holds in groups that geometrically look like the direct
product of two infinite groups.  Higher levels of this property hold
for groups that are a union of subgroups, each of which looks like a
direct product, and which can be chained together through infinite
diameter intersections. (When the chaining can be done in a nice 
group-theoretic way this leads to algebraic thickness; in this paper, 
we only require adjacent steps to have infinite 
intersection, which corresponds to a more general invariant called 
\emph{thickness}.)

As we explain below, the graph theoretic thickness of a
presentation graph completely encodes the  
thickness of the associated right-angled Coxeter
group.  Accordingly, throughout this paper, without creating any 
ambiguity, we shall use the term `thickness' to simultaneously refer to
both notions. We now
formally define thickness for graphs.

Given graphs $\Gamma_1, \Gamma_2$, the \emph{join} of
$\Gamma_1$ and $\Gamma_2$ is the graph $\Gamma_1 * \Gamma_2$ obtained
by taking the disjoint union of $\Gamma_1$ and $\Gamma_2$ and adding a
complete bipartite graph between them.  In other words,
$V(\Gamma_1*\Gamma_2):=V(\Gamma_1)\sqcup V(\Gamma_2)$ and
$E(\Gamma_1 *\Gamma_2):= E(\Gamma_1)\cup E(\Gamma_2)\cup\{v_1v_2:
\ v_1\in V(\Gamma_1), v_2\in V(\Gamma_2)\}$.  Note that the RACG of a 
join of two graphs is the direct product of  the RACGs associated to the two factors of the join. So for example the square $C_4$, which is the join of two non-edges, corresponds to the direct product of two free groups of rank two on generators of order $2$.

With this in hand, we now give the base case of our inductive definition of thickness in graphs, thickness of order $0$. By ~\cite[Proposition~2.11]{BehrstockHagenSisto:coxeter} this graph 
theoretic notion corresponds to the associated RACG being algebraically thick of 
order $0$.
\begin{definition}[Thickness of order $0$]	
	A graph $\Gamma$ is \emph{thick of order $0$} if there exists a 	partition $V(\Gamma)=A \sqcup B$ of its vertex set for which neither  	$\Gamma[A]$ nor $\Gamma[B]$ is a clique and such that 	$\Gamma=\Gamma[A]*\Gamma[B]$. 
	
	Identifying vertex sets in $\Gamma$ with the corresponding induced subgraphs, we say that a subset $S\subseteq V(\Gamma)$ is \emph{thick of order $0$} if the induced subgraph $\Gamma[S]$ is thick of order $0$. Further, we say that $S$ is a \emph{maximal} thick of order $0$ subset if for every $S'$ with  $S\subsetneq S'\subseteq V(\Gamma)$ we have that $S'$ is not thick of order $0$. 
\end{definition}
\noindent Thus a subset $S\subset V(\Gamma)$ is thick of order 0 if and only if
there exists a bipartition $S=A\sqcup B$ such that the associated 
sets of non-edges  
$S_A:=A^{(2)}\cap E(\overline{\Gamma})$ and $S_B:=B^{(2)}\cap
E(\overline{\Gamma})$ are both non-empty (ensuring $\Gamma[A]$,
$\Gamma[B]$ are non-cliques) and $\Gamma[A\cup
B]=\Gamma[A]*\Gamma[B]$.  Further, $S$ is a maximal thick of order
$0$ subset if in addition for every such partition $A\sqcup B$ and
every $v\in V(\Gamma)\setminus S$, there are some $a\in A$ and $b\in 
B$ such that neither of $av$ nor $bv$ is an edge of
$\Gamma$.

To inductively define higher levels of thickness, we shall glue 
together subsets which are thick of lower orders. Accordingly, we define the following, which will form the base level for our inductive definition. 
\begin{definition}[$T_0$: the level $0$ subsets]
The \emph{level $0$ subsets of a graph $\Gamma$}, $T_0(\Gamma)$, will denote the collection of sets $S\subseteq V(\Gamma)$ such that $S$ is a maximal thick of order $0$ subset.  Given $S\in T_0(\Gamma)$, we refer to the
associated set of non-edges $S^{(2)}\cap E(\overline{\Gamma})$ as the \emph{level $0$ components} of $\Gamma$.
\end{definition}

The level $0$ components will be our building blocks to construct the first of
a sequence of auxiliary graphs $T_k(\Gamma)$,
$k\in \mathbb{Z}_{\geq 1}$, see Remark~\ref{remark: T1} below.  Each of these graphs will have as its
vertex set the collection $E(\overline{\Gamma})$ of non-edges of
$\Gamma$, with connected
components merging as we increase the value of $k$ according to specific rules, which we specify below.

The first of these auxiliary graphs is known as the \emph{square graph}, as it encodes the induced squares of $\Gamma$ and their pairwise interactions.  
\begin{definition} [$T_1$: the square graph]\label{def:square graph}
Given a graph $\Gamma$, the \emph{square graph of $\Gamma$}, denoted
by $T_1(\Gamma)$, is defined as follows.  The vertex set of
$T_1(\Gamma)$ is the collection $E(\overline{\Gamma})$ of non-edges of
$\Gamma$.  The edges of $T_1(\Gamma)$ consist of those pairs of
non-edges $f, f'\in E(\overline{\Gamma})$ such that the set 
of vertices $f\cup f'$ induces a copy of the \emph{square} $C_4$ in
$\Gamma$. We refer to connected components in $T_1(\Gamma)$ as square components, or level $1$ components.  
\end{definition}
\begin{remark}\label{remark: T1}
An equivalent definition of the square graph of $\Gamma$ is to let $T_1(\Gamma)$ be the graph with vertex set $E(\overline{\Gamma})$ whose edges are obtained by adding for each level $0$ subset $S$
a clique  $S^{(2)}\cap E(\overline{\Gamma})$ (replacing multi-edges by single edges as necessary), whence the connection to $T_0(\Gamma)$.
\end{remark}
\begin{remark}
For certain applications in geometric group theory, it is natural to work with another auxiliary graph closely related to (but distinct from) the square graph $T_1(\Gamma)$, namely the line graph of $T_1(\Gamma)$, denoted by $S(\Gamma)$, see~\cite{behrstock2018global,DaniThomas:divcox}.  This other graph
(which in~\cite{behrstock2018global,DaniThomas:divcox} is referred to as the square graph) has the induced squares of $\Gamma$ as its vertices, and as its edges those pairs of induced squares having a diagonal in common.  

The square graph given in Definition~\ref{def:square graph} was
introduced by Behrstock, Falgas--Ravry and Susse
in~\cite{behrstock2022square} as a more natural object to study from a
combinatorial viewpoint.  As noted
in~\cite[Remark~1.2]{behrstock2022square} the two non-equivalent 
definitions of a square
graph carry essentially the same information, so working with one
rather the other is primarily a question of which definition is convenient for the
application one has in mind.
\end{remark}
The study of the properties of $S(\Gamma)$ and
$T_1(\Gamma)$ when $\Gamma$ is a random graph is known as \emph{square
percolation}, by analogy with the well-studied \emph{clique
percolation} model from network science mentioned in the introduction.

Having defined the level $0$ subsets and the level $1$ graph, we will
define the level $k$ graphs inductively after first giving a 
preliminary definition.

\begin{definition}[Suspension]
	Given $f=\{u_1, u_2\}\in E(\overline{\Gamma})$, we let
	$\mathrm{susp}(f):=\{v\in V(\Gamma):\ v \in N_{\Gamma}(u_1)\cap
	N_{\Gamma}(u_2) \}$ denote the collection of common neighbors of
	the endpoints of $f$, or, equivalently, the collection of all
	vertices $v$ such that the $3$--set of vertices $\{v\}\cup f$ induces a copy of the cherry $K_{1,2}$ in $\Gamma$.  We refer to $\mathrm{susp}(f)$ as the
	suspension based at $f$.
\end{definition}
\noindent We say a (sub)graph $\Gamma$ is a \emph{suspension} (sub)graph if there exists a 
	partition $V(\Gamma)=A \sqcup B$ of its vertex set for which 
	$|A|=2$, $\Gamma[A]$ is a non-edge, $B\neq\emptyset$ and  $\Gamma=\Gamma[A]*\Gamma[B]$. 
	If additionally $\Gamma[B]$ is a clique on one or more vertices, then we call $\Gamma$ a 
	\emph{strip} (sub)graph.

Using suspensions, we shall define a notion of \emph{support} for a component $C$ in the level $k$ graph $T_k(\Gamma)$ for $k\geq 1$. Recall that $T_k(\Gamma)$ is a graph on $E(\overline{\Gamma})$, i.e.\ the vertices of $T_k(\Gamma)$ are non-edges of $\Gamma$. The support of a component $C$ of $T_k(\Gamma)$ will be a set of vertices of $\Gamma$, $\mathrm{supp}_k(C)\subseteq V(\Gamma)$, that generalizes the notion of a level $0$ subset and intuitively corresponds to `the set of vertices of $\Gamma$ inside which $C$ lives'; more specifically, the support will be the union of all vertices from $V(\Gamma)$ that are either contained inside some non-edge $e\in E(\overline{\Gamma})\cap C$  (viewing $e$ as a set of two vertices) or contained inside a suspension based on such a non-edge $e$.

Generalizing our definition of the square graph which encoded 
intersection patterns of components of 
$T_{0}$, we inductively define 
the level $k+1$ graph $T_{k+1}(\Gamma)$ by merging components in the level $k$ graph if 
there is a non-edge of $\Gamma$ in the intersection of their supports; more formally:
\begin{definition}[$T_{k+1}$, $k\geq 1$: the level $k+1$ graph]
	Let $\Gamma$ be a graph. For $k\geq 1$, given a connected component $C$ in $T_{k}$, we define its
	\emph{support} to be
	\[\mathrm{supp}_{k}(C):=\bigcup_{f\in C}f\cup \mathrm{susp}(f),\]
	i.e.\ the collection of vertices in $V(\Gamma)$ that either belong
	to a member $f$ of $C$ or to a suspension based at some $f\in C$.
	Further, we define the latch-set of $C$ to be
	$\mathrm{latch}_{k}(C):= E(\overline{\Gamma})\cap
	\left(\mathrm{supp}_{k}(C)\right)^{(2)}$.

We  then define $T_{k+1}(\Gamma)$ as a graph on the 
	vertex set 
	$E(\overline{\Gamma})$ (the non-edges of $\Gamma$) 
	by joining $f_1, f_2\in
	E(\overline{\Gamma})$ by an edge if for $i\in\{1,2\}$ there exist
	connected components $C_1$, $C_2$  in $T_k$ such that $f_i \in
	\mathrm{latch}_{k}(C_i)$ and
	$\mathrm{latch}_{k}(C_1)\cap\mathrm{latch}_{k}(C_2) \neq
	\emptyset$.

\end{definition}

\begin{figure}
	\centering
	\tikzset{every picture/.style={line width=0.75pt}} 
	\begin{tikzpicture}[x=0.75pt,y=0.75pt,yscale=-1,xscale=1]
		
		\coordinate (A) at (118,88); 
 		 \coordinate (B) at (180,88); 
		\coordinate (C) at (180,103.5);
		\coordinate (D) at (180,57);
		\coordinate (E) at (180,72.5);
		\coordinate (F) at (180,41.5);
		\coordinate (G) at (118,57);
 		 \draw[color=black,fill=black] (A) circle[radius=0.05cm];
 		 \draw[color=black,fill=black] (B) circle[radius=0.05cm];
		  \draw[color=black,fill=black] (C) circle[radius=0.05cm];
 		 \draw[color=black,fill=black] (D) circle[radius=0.05cm];
		  \draw[color=black,fill=black] (E) circle[radius=0.05cm];
 		 \draw[color=black,fill=black] (F) circle[radius=0.05cm];
		  \draw[color=black,fill=black] (G) circle[radius=0.05cm];
		  
		\draw    (118,88) -- (180,88) ;
		\draw    (118,88) -- (180,103.5) ;
		\draw    (118,88) -- (180,57) ;
		\draw    (118,88) -- (180,72.5) ;
		\draw    (118,88) -- (180,41.5) ;
		\draw    (118,57) -- (180,41.5);
		\draw    (118,57) -- (180,57) ;
		\draw    (118,57) -- (180,88) ;
		\draw    (118,57) -- (180,103.5) ;
		\draw    (118,57) -- (180,72.5) ;

		 \coordinate (A1) at (270,57); 
 		 \coordinate (B1) at (300,57); 
		\coordinate (C1) at (270,88);
		\coordinate (D1) at (300,88);
		\coordinate (E1) at (330,57);
		\coordinate (F1) at (330,88);
		\coordinate (G1) at (360,57);
		\coordinate (I1) at (360,88);
		\coordinate (J1) at (390,57);
		\coordinate (H1) at (390,88);
		\coordinate (K1) at (420,57);
		\coordinate (L1) at (420,88);
		
 		 \draw[color=black,fill=black] (A1) circle[radius=0.05cm];
 		 \draw[color=black,fill=black] (B1) circle[radius=0.05cm];
		  \draw[color=black,fill=black] (C1) circle[radius=0.05cm];
 		 \draw[color=black,fill=black] (D1) circle[radius=0.05cm];
		  \draw[color=black,fill=black] (E1) circle[radius=0.05cm];
 		 \draw[color=black,fill=black] (F1) circle[radius=0.05cm];
		  \draw[color=black,fill=black] (G1) circle[radius=0.05cm];
		   \draw[color=black,fill=black] (I1) circle[radius=0.05cm];
 		 \draw[color=black,fill=black] (J1) circle[radius=0.05cm];
		  \draw[color=black,fill=black] (H1) circle[radius=0.05cm];
 		 \draw[color=black,fill=black] (K1) circle[radius=0.05cm];
		  \draw[color=black,fill=black] (L1) circle[radius=0.05cm];
		
		\draw    (270,57) -- (300,57) ;
		\draw    (300,57) -- (270,88) ;
		\draw    (270,57) -- (300,88) ;
		\draw    (270,88) -- (300,88) ;
		\draw    (300,57) -- (330,57) ;
		\draw    (330,57) -- (300,88) ;
		\draw    (300,57) -- (330,88) ;
		\draw    (300,88) -- (330,88) ;
		\draw    (330,57) -- (360,57) ;
		\draw    (360,57) -- (330,88) ;
		\draw    (330,57) -- (360,88) ;
		\draw    (330,88) -- (360,88) ;
		\draw    (360,57) -- (390,57) ;
		\draw    (390,57) -- (360,88) ;
		\draw    (360,57) -- (390,88) ;
		\draw    (360,88) -- (390,88) ;
		\draw    (390,57) -- (420,57) ;
		\draw    (420,57) -- (390,88) ;
		\draw    (390,57) -- (420,88) ;
		\draw    (390,88) -- (420,88) ;
		
		\draw (140,110) node [anchor=north west][inner sep=0.75pt]   [align=left] {$K_{2,5}$};
		\draw (290,110) node [anchor=north west][inner sep=0.75pt]   [align=left] {A `path of squares'};
	\end{tikzpicture}
	\caption{examples of a thick of order $0$ graph (left) and a thick of order $1$ graph (right).}
     \label{fig 1}
\end{figure}

\begin{figure}
	\centering
	\tikzset{every picture/.style={line width=0.75pt}} 
	
	\begin{tikzpicture}[x=0.75pt,y=0.75pt,yscale=-1,xscale=1]
		
		 \coordinate (A1) at (175,65); 
 		 \coordinate (B1) at (225,65); 
		\coordinate (C1) at (175,100);
		\coordinate (D1) at (225,100);
		\coordinate (E1) at (275,65);
		\coordinate (F1) at (275,100);
		\coordinate (G1) at (325,65);
		\coordinate (I1) at (325,100);
		\coordinate (J1) at (375,65);
		\coordinate (H1) at (375,100);
		\coordinate (K1) at (425,65);
		\coordinate (L1) at (425,100);
		\coordinate (M1) at (225,140);
		\coordinate (N1) at (275,140);
		\coordinate (O1) at (325,140);
		\coordinate (P1) at (375,140);
		\coordinate (S1) at (175,140);

 		 \draw[color=black,fill=black] (A1) circle[radius=0.05cm];
 		 \draw[color=black,fill=black] (B1) circle[radius=0.05cm];
		  \draw[color=black,fill=black] (C1) circle[radius=0.05cm];
 		 \draw[color=black,fill=black] (D1) circle[radius=0.05cm];
		  \draw[color=black,fill=black] (E1) circle[radius=0.05cm];
 		 \draw[color=black,fill=black] (F1) circle[radius=0.05cm];
		  \draw[color=black,fill=black] (G1) circle[radius=0.05cm];
		   \draw[color=black,fill=black] (I1) circle[radius=0.05cm];
 		 \draw[color=black,fill=black] (J1) circle[radius=0.05cm];
		  \draw[color=black,fill=black] (H1) circle[radius=0.05cm];
 		 \draw[color=black,fill=black] (K1) circle[radius=0.05cm];
		  \draw[color=black,fill=black] (L1) circle[radius=0.05cm];
		    \draw[color=black,fill=black] (M1) circle[radius=0.05cm];
 		 \draw[color=black,fill=black] (N1) circle[radius=0.05cm];
		  \draw[color=black,fill=black] (O1) circle[radius=0.05cm];
 		 \draw[color=black,fill=black] (P1) circle[radius=0.05cm];
		  \draw[color=black,fill=black] (S1) circle[radius=0.05cm];

		\draw    (175,65) -- (225,65) ;
		\draw    (225,65) -- (175,100) ;
		\draw    (175,65) -- (225,100) ;
		\draw    (175,100) -- (225,100) ;
		\draw    (225,65) -- (275,65) ;
		\draw    (275,65) -- (225,100) ;
		\draw    (225,65) -- (275,100) ;
		\draw    (225,100) -- (275,100) ;
		\draw    (275,65) -- (325,65) ;
		\draw    (325,65) -- (275,100) ;
		\draw    (275,65) -- (325,100) ;
		\draw    (275,100) -- (325,100) ;
		\draw    (325,65) -- (375,65) ;
		\draw    (375,65) -- (325,100) ;
		\draw    (325,65) -- (375,100) ;
		\draw    (325,100) -- (375,100) ;
		\draw    (375,65) -- (425,65) ;
		\draw    (425,65) -- (375,100) ;
		\draw    (375,65) -- (425,100) ;
		\draw    (375,100) -- (425,100) ;
		\draw    (175,100) -- (225,140) ;
		\draw    (375,100) -- (375,140) ;
		\draw    (175,100) -- (275,140) ;
		\draw    (175,100) -- (325,140) ;
		\draw    (175,100) -- (175,140) ;
		\draw    (175,100) -- (375,140) ;
		\draw    (375,100) -- (325,140) ;
		\draw    (375,100) -- (275,140) ;
		\draw    (375,100) -- (225,140) ;
		\draw    (375,100) -- (175,140) ;
		
		\draw (180,160) node [anchor=north west][inner sep=0.75pt]   [align=left] {A thick union of $K_{2,5}$ and a path of squares};
	\end{tikzpicture}\caption{an example of a thick of order $2$ graph}
	\label{fig 2}
\end{figure}

\begin{definition}[Thickness]\label{def: thickness}
	Given a collection of non-edges $C\subseteq E(\overline{\Gamma})$
	that forms a level $k$ connected component for some $k\geq 1$ but
	is not a subset of a level $0$ component of $\Gamma$, we say that
	$C$ is a level $k_0$ component if $k_0$ is the least integer
	$k\geq 1$ such that $C$ forms a connected component in $T_k$.  We
	say that such a component has \emph{full support} if
	$\mathrm{supp}_{k_0}(C)=V(\Gamma)$ or, equivalently, if
	$\mathrm{latch}_{k_0}(C)=E(\overline{\Gamma})$.  If this occurs,
	we say that the graph $\Gamma$ is \emph{thick of order $k_0$}.
	More generally, we say that $\Gamma$ is \emph{thick} if it is
	thick of some finite order $k\geq 0$.
\end{definition}
\begin{remark}[Non-monotonicity under the addition/removal of edges]\label{remark: nonmonotonicity}
	The thickness properties of the graphs $T_k(\Gamma)$ are \emph{not} monotone under the addition of edges to $\Gamma$. Indeed, adding an edge to $\Gamma$ could potentially create some new induced squares in $\Gamma$  (and thus new edges in $T_1(\Gamma)$) as well as destroy existing ones. As an example, consider the complete bipartite graph $\Gamma=K_{2,n-2}$ with one part of size $2$ and another of size $n-2$. This graph is clearly thick of order $0$. However if we delete any edge from $\Gamma$, or if we add an edge inside the part of size $2$, then $\Gamma$ ceases to be thick of any order.
\end{remark}

The non-monotonicity of thickness under the addition or 
deletion of edges  make its study a delicate matter, and explains in particular why many results on the geometric properties of random RACG feature both a lower and an upper bound on the edge probability $p$ required for a certain property to hold a.a.s..

To compute the order of thickness of RACGs, the following definition of
hypergraph index was introduced by Levcovitz
in~\cite{Levcovitz:QIinvariantThickness}.  Levcovitz proved
in~\cite[Theorem~A]{Levcovitz:DivergenceRACG} that a graph $\Gamma$
has finite hypergraph index if and only if the corresponding RACG
$W_{\Gamma}$ is thick of order $k$; moreover, this
occurs if and only if the RACG $W_{\Gamma}$ has divergence which is
polynomial of degree $k+1$. This strengthens earlier partial results 
from \cite{BDM,  DaniThomas:divcox, BehrstockHagenSisto:coxeter, behrstock2018global}.

\begin{definition}[Hypergraph index]
Let $\Gamma$ be a graph.  Let $\Psi(\Gamma)$ denote the collection of
subsets of $V(\Gamma)$ that induce strip subgraphs in $\Gamma$.  Let
$\Lambda_{0}$ be the (non-uniform) hypergraph whose vertex set is
$V(\Gamma)$ and whose set of hyperedges is $T_{0}(\Gamma)\cup
\Psi(\Gamma)$.

For all integers $n\geq 0$, inductively define a hypergraph
$\Lambda_{n+1}$ as follows.  Introduce an equivalence relation
$\cong_n$ on pairs of hyperedges of $\Lambda_{n}$ by setting
$E\equiv_{n}E'$ when there exists a finite sequence of hyperedges
$E=E_{1}$, $E_2$, $\ldots$, $E_k$, $E_{k+1}=E'$ such that for every
$i\in [k]$ the subgraph of $\Gamma$ induced by $E_{i}\cap E_{i+1}$
contains a non-edge.  The hypergraph $\Lambda_{n+1}$ then has
$V(\Gamma)$ as its vertex set, and a hyperedge $\bigcup_{E\in C}E$ for
each $\equiv_n$--equivalence class $C$.

	If $T_{0}(\Gamma)\neq\emptyset$ we define the \emph{hypergraph
	index} of $\Gamma$ to be the smallest non-negative integer $n$ for
	which $\Lambda_n$ contains $V(\Gamma)$ as a hyperedge.  When no
	such $n$ exists or when $T_{0}(\Gamma)=\emptyset$ we define the
	\emph{hypergraph index} of $\Gamma$ to be $\infty$.
\end{definition}
\noindent Our definition of higher order thickness for graphs is defined so 
that it encodes this notion. Indeed, it is not hard to see that the hypergraph index of Levcovitz coincides precisely with the order of thickness as defined in
Definition~\ref{def: thickness}.  Thus Levcovitz's results imply the 
order of thickness of a RACG $W_{\Gamma}$ coincides precisely with 
the order of 
graph theoretic thickness of its presentation graph $\Gamma$, and
the two notions can be conflated in the context of RACGs.

\section{An extremal result for thickness in graphs}\label{section: extremal result}
We shall prove the following strengthening of Theorem~\ref{theorem: combinatorial characterisation of thickness} using an inductive strategy. 
		Recall that the union $\Gamma_1\cup \Gamma_2$ of the graphs $\Gamma_1$ and $\Gamma_2$ is the graph $\Gamma$ with $V(\Gamma)=V(\Gamma_1)\cup V(\Gamma_2)$ and $E(\Gamma)=E(\Gamma_1)\cup E(\Gamma_2)$. 
		

		\begin{theorem}\label{theorem: strong form of extremal theorem}
			Let $k\in \mathbb{Z}_{\geq 0}$.  Let $\Gamma_1=(V_{1},
			E_1)$ be a graph, and let $\Gamma_2=(V_2, E_2)$ be 
			either a copy of the cherry $K_{1,2}$ or 
			a thick of order at most $k$ graph.
			Set $I:= V_1\cap V_2$, and suppose that $\Gamma_2[I]$ is
			not a clique of order at most $2$.  Then
			\begin{align}\label{eq: bound on union with k-thick, inductive statement} e(\Gamma_1\cup \Gamma_2)\geq e(\Gamma_1)+2\vert V_2\setminus V_1\vert.\end{align}
			In particular if $\Gamma$ is 
			thick of order at most $k+1$, then 
			$e(\Gamma)\geq 2v(\Gamma)-4$.

		\end{theorem}
		\begin{remark}\label{remark: cherry easy}
			If $\Gamma_2\cong K_{1,2}$, then~\eqref{eq: bound on union
			with k-thick, inductive statement} is easily seen to hold:
			if $\Gamma_2[I]$ is a non-clique, then it must either be
			the entirety of $\Gamma_{2}$, or it must consist of the
			unique non-edge in $\Gamma_2$, and in both cases the
			claimed upper bound holds.  Thus the content
			of Theorem~\ref{theorem: strong form of extremal theorem}
			lies in the case where $\Gamma_2$ is thick of order at
			most $k$; the formulation including the cherry as a special case is 
			nonetheless useful as it will make the formulation of our inductive argument easier.
		\end{remark}
		\noindent Our proof of Theorem~\ref{theorem: strong form of extremal theorem} follows an inductive strategy that relies on the following characterization of thick of order $k+1$ graphs.
		\begin{proposition}\label{prop: characterisation of thickness}
			Let $\Gamma$ be a thick of order $k+1$ graph, for some $k\geq 0$. Then for some  $T>0$,  there exists a collection of induced subgraphs of $\Gamma$,  $\Gamma_i=(V_i, E_i)$ for $i\in [T]$, and a tree $\mathcal{T}$ on the vertex set $[T]$ such that:
			\begin{enumerate}[(a)]
				\item for every $i\in [T]$, $\Gamma_i$ is a copy of the cherry $K_{1,2}$ 
				or is thick of
				order at most $k$;
				\item for every $ij \in E(\mathcal{T})$,  the induced subgraph
				$\Gamma[V_i\cap V_j]$ is not complete; 
				\item $\bigcup_{i=1}^T V_i =V$.
			\end{enumerate}
		\end{proposition}
		
			\begin{proof} Follows immediately from the combinatorial
			characterization of thickness in right-angled Coxeter
			groups given in~\cite[Theorem
			II]{BehrstockHagenSisto:coxeter} which implies that any 
			thick graph gives a connected 
			graph whose vertices are associated to full induced 
			subgraphs which are thick of lower order and which are 
			connected by an edge if the associated graphs overlap in 
			a non-clique. The result then follows from the fact that every
			connected graph contains a spanning tree as a subgraph.
		\end{proof}
		
		\begin{proposition}\label{prop: in particular part}
			Let $k\in \mathbb{Z}_{\geq 0}$. Suppose that every thick 
			of order at most $k$ graph $\Gamma$ satisfies 
			$e(\Gamma)\geq 2v(\Gamma) - 4$ and that in
			addition~\eqref{eq: bound on union with k-thick, inductive
			statement} holds for all graphs $\Gamma_1$ and all thick
			of order at most $k$ graphs $\Gamma_2$.  Then every
			graph $\Gamma'$ that is thick of order at most
			$k+1$ satisfies 
			$e(\Gamma')\geq 2v(\Gamma')-4$.
		\end{proposition}

		\begin{proof}			
			Let $\Gamma'$ be a thick of order at most $k+1$ graph, 
			and let $\Gamma'_i=(V_i, E_i)$, $i\in [T]$ be the
			collection of induced thick of order at most $k$ subgraphs
			of $\Gamma'$ whose existence is guaranteed by
			Proposition~\ref{prop: characterisation of thickness}.
			Since $\mathcal{T}$ is a tree, we can relabel the indices
			of the $\Gamma'_i$ so that for every $j>1$ there exists
			$i\in [j-1]$ with $ij \in E(\mathcal{T})$.

			By our assumption on graphs which are thick of order at most
			$k$, we have $e(\Gamma'_1)\geq 2\vert V_1\vert -4$.  Now
			for each $j>1$, there exists $i\in [j-1]$ such that $ij\in
			E(\mathcal{T})$ and hence (by Proposition~\ref{prop:
			characterisation of thickness}.(b)) $\Gamma'[V_i\cap
			V_j]$ is non-complete.  Since $\Gamma'_j$ is thick of
			order at most $k$ or a copy of the cherry $K_{1,2}$, our assumption allows us to 
			apply~\eqref{eq: bound on union
			with k-thick, inductive statement}, which yields:
			\begin{align*}
				e\left(\left(\bigcup_{j'<j}\Gamma'_{j'}\right)\cup 
				\Gamma'_j \right)\geq e\left(\bigcup_{j'<j}\Gamma'_{j'}\right)+ 2\left\vert V_j \setminus \left(\bigcup_{j'<j} V_{j'}\right)\right\vert.
			\end{align*}
			Iterating $T-1$ times, we get the desired bound on $e(\Gamma')$:
			\begin{align*}
				e(\Gamma')&\geq e\left(\bigcup_{j=1}^T 
				\Gamma'_i\right)\geq 2\vert V_1\vert -4 + 
				\sum_{j>1}2\left\vert V_j \setminus 
				\left(\bigcup_{j'<j} V_{j'}\right)\right\vert= 2\left\vert 
				\bigcup_{j=1}^T V_i\right\vert -4= 2v(\Gamma') -4,
			\end{align*}
			where the last equality follows from
			Proposition~\ref{prop: characterisation of thickness}.(c).
		\end{proof}
		
		\noindent Thus armed, we begin our proof of Theorem~\ref{theorem: strong
		form of extremal theorem} by proving the base case
		$k=0$.

		\begin{proposition}\label{prop: thick of order 0 bound}
			Let $\Gamma$ be thick of order $0$, and let 
			$V(\Gamma)=A\sqcup B$ be a partition of its vertex-set such that neither of $\Gamma[A]$ nor $\Gamma[B]$ is a clique and
			$\Gamma=\Gamma[A]* \Gamma[B]$.
Setting $a:=\vert A \vert$ and $m:=v(\Gamma)$, we have 
\[e(\Gamma)\geq a(m-a)\geq 2(m-2)=2m-4.\]
		\end{proposition}
		\begin{proof}
	As neither of $\Gamma[A]$ and $\Gamma[B]$ is a clique, we have $2\leq a \leq m-2$. Now clearly $e(\Gamma)\geq \vert A\vert \cdot \vert B\vert =a(m-a)$, which for $a\in [2,m-2]$ is at least $2(m-2)$.
		\end{proof}

		\begin{proposition}\label{proposition: base case}
			The statement of Theorem~\ref{theorem: strong form of extremal theorem} holds for $k=0$.
		\end{proposition}
		\begin{proof}
			We have already established in Proposition~\ref{prop: 
			thick of order 0 bound} that a thick of order~0 graph on $m$ 
			vertices must have at least $2m-4$ edges. By 
			Proposition~\ref{prop: in particular part}, it is thus enough to establish~\eqref{eq: bound on union with k-thick, inductive statement} for $k=0$. Further, by Remark~\ref{remark: cherry easy} we need not consider the case where $\Gamma_2$ is a cherry.

			Let $\Gamma_1=(V_1, E_1)$ be a graph, and let
			$\Gamma_2=(V_2, E_2)$ be a thick of order $0$ graph.  Set
			$I:= V_1\cap V_2$, and suppose this is a non-empty subset
			of $V_2$.  If $V_2\subseteq V_1$, then~\eqref{eq: bound on
			union with k-thick, inductive statement} holds trivially.
			We may thus assume $V_2\setminus V_1$ is non-empty.  
			Since $\Gamma_2$ is thick of order zero it admits a 
			partition $V(\Gamma_2)=A\sqcup B$ where 
			$\Gamma_2[A]$ and $\Gamma_2[B]$ are both non-complete and
			$\Gamma_2=\Gamma_2[A]* \Gamma_2[B]$.
			Denote $a:=\vert A\vert$, $b:=\vert B\vert$, $x:=\vert A\cap
			I\vert$ and $y:=\vert B\cap I\vert$.  Assume without loss
			of generality that $x\leq y$.

			Then we have the following
			key inequality:		
			\begin{align}
				e(\Gamma_1\cup \Gamma_2)-e(\Gamma_1)-2\vert V_2\setminus V_1\vert &\geq  \vert B\setminus I\vert \left(\vert A\vert -2 \right)+ \vert A\setminus I\vert \left( \vert B\cap I\vert -2\right)\notag \\ &= (b-y)(a-2)+(a-x)(y-2)\label{eq: thick of order zero intersection}.
			\end{align}
	\noindent To see why this inequality holds note the following: every vertex in $A$ sends edges to
	every vertex in $B\setminus I$, giving us a first set of $\vert
	B\setminus I\vert\cdot \vert A\vert$ edges; every vertex in
	$B\cap I$ sends edges to every vertex in $A\setminus I$, giving us
	a second set of $\vert A\setminus I\vert \cdot \vert B\cap I\vert
	$ edges disjoint from the first; lastly, none of the edges in
	these two sets belong to $\Gamma_1$ since all of them are incident
	to a vertex of $(A\sqcup B)\setminus I=V(\Gamma_2)\setminus
	V(\Gamma_1)$.

			To conclude the proof, it suffices to show that either $\Gamma_2[I]\cong K_{\epsilon}$ for some $\epsilon\in\{1,2\}$ or that the expression on the right hand side of~\eqref{eq: thick of order zero intersection} is non-negative.  By thickness of order $0$, we know that $a=\vert A\vert$ and $b=\vert B\vert$ are both at least $2$, and we also have $a\geq x$ and $b\geq y$.  In particular the first term in~\eqref{eq: thick of order zero intersection} above is always non-negative.

			Consider now the second term. We have $2y \geq x+y = \vert I\vert >0$, whence $y>0$. If $y=1$ and $x=0$, then $\Gamma_2[I]\cong K_1$ and we are done.
			Similarly if $y=1$ and $x=1$, then $\Gamma_2[I]\cong K_2$ (since $\Gamma_2$ contains as a subgraph the complete bipartite graph on the bipartition $A\sqcup B$) and we are done.
			Finally if $y\geq 2$, then $(b-y)(a-2)+(a-x)(y-2) \geq 0$, as desired. The proposition follows.
		\end{proof}
		\begin{proof}[Proof of Theorem~\ref{theorem: strong form of extremal theorem}]
			We perform induction on $k$. We established the base case 
			$k=0$ in Proposition~\ref{proposition: base case}. For 
			the inductive step, assume we have proved the theorem for 
			all $k\leq K$, for some $K\geq 0$. By 
			Proposition~\ref{prop: in particular part}, it suffices then to show~\eqref{eq: bound on union with k-thick, inductive statement} holds for $k=K+1$.

			Let $\Gamma_1=(V_1, E_1)$ be an arbitrary graph, and let $\Gamma_2=(V_2, E_2)$ be a thick of order $K+1$ graph. Set $I= V_1\cap V_2$, and suppose this is a non-empty subset of $V_2$ with $\Gamma_2[I]\not \cong K_1$, $K_2$. Note first of all that if $\vert V_2 \setminus V_1\vert =0$, then~\eqref{eq: bound on union with k-thick, inductive statement} holds trivially. We may thus assume that $V_2\setminus V_1$ is non-empty.

			Applying Proposition~\ref{prop: characterisation of 
			thickness} to $\Gamma_2$, provides a collection of induced subgraphs $\Gamma_{2,i}=(V_{2,i}, E_{2,i})$, $i\in [T]$,  of our graph $\Gamma_2$, each of which is either a $K_{1,2}$ or thick of order at most $K$, together with a tree $\mathcal{T}$ on $[T]$ satisfying properties (a)---(c) from the statement of Proposition~\ref{prop: characterisation of thickness}. Here we must consider two cases.

			\noindent \textbf{Case 1:}  suppose first of all that there is some $i_0$ such that 
			\begin{align}\label{eq: good condition on union}
				e(\Gamma_1\cup \Gamma_{2,i_0})\geq e(\Gamma_1) +2\vert V_{2,i_0}\setminus V_1\vert.
			\end{align}
			Reordering the indices of the $\Gamma_{2,i}$ as necessary, we may assume that $i_0=1$ and that for every $j>1$ there exists $i\in [j-1]$ with $ij\in E(\mathcal{T})$  (the existence of such an ordering is implied by the fact that $\mathcal{T}$ is a tree), which in turn implies $\Gamma_2[V_{2,i}\cap V_{2,j}]$ is non-complete (by property (b)). In particular, we must have that $\Gamma_2\left[\left(V_1\cup \bigcup_{i<j}V_{2,i}\right)\cap V_{2,j}\right]$ is a graph on at least two vertices not isomorphic to $K_1$ or $K_2$. Applying our inductive hypothesis $T-1$ times and appealing to~\eqref{eq: good condition on union}, we conclude that
			\begin{align*}
				e\left(\Gamma_1\cup \Gamma_2\right)\geq e\left(\Gamma_1\cup \left(\bigcup_j \Gamma_{2,j}\right)\right)
				&\geq e\left(\Gamma_1\cup \Gamma_{2,1}\right) +\sum_{j>1}  2\left \vert V_{2,j}\setminus \left(V_1\cup \left(\bigcup_{j'<j} V_{2,j'}\right)\right)\right \vert \\
				&\geq e(\Gamma_1) + 2\left\vert \bigcup_{j\geq 1} V_{2,j}\setminus V_1
				\right\vert= e(\Gamma_1)  +2\vert V_2\setminus V_1\vert,
			\end{align*}
			and~\eqref{eq: bound on union with k-thick, inductive statement} holds as required.

			\noindent \textbf{Case 2:}  suppose that~\eqref{eq: good condition on union} does not hold for any $i_0\in [T]$. By our inductive hypothesis this implies the following:
			\begin{align*}
				(\star)  && \textrm{for every }i\in [T] \textrm{, }\Gamma_{2}[V_1\cap V_{2,i}]\textrm{ is a clique on at most two vertices}
			\end{align*}
			By property (c), there exists some $i_0$ such that $V_{2, i_0}\cap V_1$ is a non-empty subset of $I=V_1\cap V_2$. Reordering the indices of the $\Gamma_{2,i}$ as necessary, we may assume that $i_0=1$, that $\vert V_1\cap V_{2, 1} \vert \geq \vert V_1\cap V_{2,i}\vert $ for all $i\in [T]$, and that for every $j>1$ there exists $i\in [j-1]$ with $ij\in E(\mathcal{T})$.

			Now $\Gamma_2[V_{2,1}\cap V_1]=\Gamma_{2,1}[V_{2,1}\cap V_1]$, which by $(\star)$ above is a clique on at least one and at most two vertices. Since  $\Gamma_2[V_{2}\cap V_1]$ is not a clique on at most two vertices, it follows by property (c) again that there is some $j>1$ such that $\left(V_{2,j}\cap V_1 \right)\setminus \left(V_{2,1}\cap V_1 \right)$ is non-empty. Let $j_0$ be the least such $j$.

			Consider now the graph $\Gamma_{2, [j_0]}:=\bigcup_{i=1}^{j_0} \Gamma_{2,i}$, and write $V_{2,[j_0]}:=\bigcup_{i=1}^{j_0}V_{2,i}$ for its vertex set. Since for every $j>1$ there exists $i\in [j-1]$ with $ij\in E(\mathcal{T})$, it follows from property (b) and the definition of thickness that $\Gamma_{2, [j_0]}$ is thick of order at most $K+1$. By the `in particular' part of our inductive hypothesis, we have
			\begin{align}\label{eq: bound when first circle back}
				e(\Gamma_{2, [j_0]})\geq 2\left \vert V_{2,[j_0]}\right \vert -4. 
			\end{align}
			Set $I_0:= V_1\cap V_{2, 1}$ and $I_1:=V_1\cap V_{2,j_0}$, so that $V_1\cap V_{2, [j_0]}= I_0\cup I_1$ by the minimality of $j_0$. Further we have $I_0\setminus I_1$ and $I_1\setminus I_0$ both non-empty --- indeed, the latter follows by definition of $j_0$, and the former from the maximality of $\vert V_1\cap V_{2,i}\vert$.

			Since the elements of  $I_1\setminus I_0$ do not appear in any $V_{2,i}$ with $i\in [j_0-1]$, and since the elements of $I_0\setminus I_1$ do not belong to $V_{2,j_0}$ it follows that there is no edge from $I_1\setminus I_0$ to $I_0\setminus I_1$ in $\Gamma_{2, [j_0]}$. Indeed, note the vertices in $I_1\setminus I_0$ belong to none of the $\Gamma_{2, j}$ for $j<j_0$, and so are incident to no edges in these graphs, while the vertices in $I_0\setminus I_1$ do not belong to $\Gamma_{2,j_0}$, and so cannot send edges to $I_1\setminus I_0$ in that graph.
			
			By the inclusion-exclusion principle and inequality~\eqref{eq: bound when first circle back}, we have 
			\begin{align}
				e(\Gamma_1\cup \Gamma_{2, [j_0]})&\geq 
				e(\Gamma_1)+e(\Gamma_{2, [j_0]}) - e\left(\Gamma_{2, [j_0]}[I_0\cup I_1]\right)\notag \\
				&\geq e(\Gamma_1) +2\left\vert V_{2, 
				[j_0]}\right\vert-4 - e\left(\Gamma_{2, [j_0]}[I_0\cup I_1]\right)\notag\\
				&\geq e(\Gamma_1)+ 2\left\vert V_{2, [j_0]}\right\vert-4 - \binom{\vert I_0 \cup I_1\vert}{2}+\vert I_0\setminus I_1\vert \cdot \vert I_1\setminus I_0\vert.	\label{eq: union bound when circled back}
			\end{align}

			\begin{claim}\label{claim: final nail in the extremal coffin}
				$2\vert V_{2, [j_0]}\vert -4 - \binom{\vert I_0 \cup I_1\vert}{2}+\vert I_0\setminus I_1\vert \cdot \vert I_1\setminus I_0\vert\geq 2 \vert V_{2,[j_0]}\setminus V_1\vert $.
			\end{claim}	
			\begin{proof}
				Set $t:=\vert I_0\cup I_1\vert$ and $N:= \vert V_{2,[j_0]}\vert$, so that $\vert V_{2,[j_0]}\setminus V_1\vert= N-t$.

				Since $I_0$ and $I_1$ have size at least $1$ and at most $2$ (by $(\star)$), and since $I_0\Delta I_1\neq \emptyset$ we have $2\leq t\leq 4$. Furthermore, $t=4$ is possible if and only $I_0$ and $I_1$ are disjoint sets of size two.

				If $2\leq t\leq 3$, then we have 
				\begin{align*}
					2\vert V_{2, [j_0]}\vert -4 - \binom{\vert I_0 \cup I_1\vert}{2}+\vert I_0\setminus I_1\vert \cdot \vert I_1\setminus I_0\vert\geq 2N-4 -\binom{t}{2}+1= 2(N-t)-\frac{(t-2)(t-3)}{2},
				\end{align*}
				which is equal to $2(N-t)$ as desired. On the other hand if $t=4$, then we have
				\begin{align*}
					2\vert V_{2, [j_0]}\vert -4 - \binom{\vert I_0 \cup I_1\vert}{2}+\vert I_0\setminus I_1\vert \cdot \vert I_1\setminus I_0\vert= 2N-4 -\binom{4}{2}+4= 2(N-4) +2>2(N-4),
				\end{align*}
				as required. The claim follows.
			\end{proof}
			\noindent Combining Inequality~\eqref{eq: union bound when circled back} and Claim~\ref{claim: final nail in the extremal coffin}, we deduce that
			\begin{align}\label{eq: consequence of final nail}
				e(\Gamma_1\cup \Gamma_{2, [j_0]}) \geq e(\Gamma_1)+2 \vert V_{2,[j_0]}\setminus V_1\vert.
			\end{align}
			We can now conclude the proof of this case much as we did in Case 1: for every $j>j_0$ there exists $i\in [j-1]$ with $ij\in E(\mathcal{T})$, which in turn implies $\Gamma_2[V_{2,i}\cap V_{2,j}]$ is non-complete. In particular, we must have that $\Gamma_2\left[\left(V_1\cup \bigcup_{i<j}V_{2,i}\right)\cap V_{2,j}\right]$ is a graph on at least two vertices not isomorphic to $K_1$ or $K_2$. Applying our inductive hypothesis $T-j_0$ times and appealing to~\eqref{eq: consequence of final nail}, we conclude that
			\begin{align*}
				e\left(\Gamma_1\cup \Gamma_2\right)\geq e\left(\Gamma_1\cup \left(\bigcup_j \Gamma_{2,j}\right)\right)
				&\geq e\left(\Gamma_1\cup \Gamma_{2,[j_0]}\right) +\sum_{j>j_0}  2\left \vert V_{2,j}\setminus \left(V_1\cup \left(\bigcup_{j'<j} V_{2,j'}\right)\right)\right \vert \\
				&\geq e(\Gamma_1) + 2\left\vert \bigcup_{j\geq 1} V_{2,j}\setminus V_1
				\right\vert= e(\Gamma_1)  + 2\vert V_2\setminus V_1\vert,
			\end{align*}
			and~\eqref{eq: bound on union with k-thick, inductive statement} holds as required.
			
		\end{proof}

\section{Thresholds in random right-angled Coxeter groups}\label{section: thickness}

\subsection{Relative hyperbolicity: proof of Theorem~\ref{theorem: main}}\label{section: hyperbolicity threshold}

The key novelty in this paper is in Theorem~\ref{theorem: main} part
(i) and its proof, which we shall now give.  By results of
Behrstock Hagen and Sisto~\cite{BehrstockHagenSisto:coxeter} and
Levcovitz~\cite{Levcovitz:QIinvariantThickness} discussed in the
introduction, a RACG $W_{\Gamma}$ is relatively hyperbolic if and only
if its presentation graph $\Gamma$ fails to be thick.  Thus part (i)
of Theorem~\ref{theorem: main} is implied by the following stronger
theorem giving upper bounds on the order of thick components in
$\Gamma\sim \mathcal{G}_{n,p}$.

\begin{theorem}\label{theorem: lower bound}
	Let $p=p(n)\leq \frac{1}{4\sqrt{n\log n}}$ and $\Gamma\sim 	
	\mathcal{G}_{n,p}$.  Then a.a.s.\  for every $k\in \mathbb{N}$, every component in $T_k(\Gamma)$ has support of size at most $\log n$. In particular, $\Gamma$ is a.a.s.\ not thick of order~$k$.
\end{theorem}

\noindent We prove Theorem~\ref{theorem: lower bound} in an inductive 
fashion by combining our extremal result, Theorem~\ref{theorem: combinatorial characterisation of thickness}, with the following simple result about random graphs.
\begin{proposition}\label{prop: no dense sets}
	Let $p=p(n)\leq \frac{1}{4\sqrt{n\log n}}$ and $\Gamma\sim \mathcal{G}_{n,p}$. Then a.a.s.\ for every $m \in [\log n, 2\log n]$, every $m$--vertex subset of $\Gamma$ supports at most $2m-5$ edges.
\end{proposition}
\begin{proof}
Let $X_m$ denote the number of $m$--vertex subsets in $V(\Gamma)$ supporting at least $2m-4$ edges in $\Gamma$. Set $X:=\sum_{m\in [\log n, 2\log n]}X_m$. Then we have for all $n$ sufficiently large

\begin{align*}
\mathbb{E}(X)&=\sum_{m=\lceil \log n\rceil}^{\lfloor 2\log n\rfloor } X_m\leq \sum_{m=\lceil \log n\rceil}^{\lfloor 2\log n\rfloor }  \binom{n}{m}\binom{\binom{m}{2}}{2m-4}p^{2m-4} \\
&\textrm{(union bound over all $m$-sets and all possible choices of $2m-4$ edges inside an $m$-set)}\\
&\leq \sum_{m=\lceil \log n\rceil}^{\lfloor 2\log n\rfloor } \left(\frac{ne}{m}\right)^m  \left(\frac{3mp}{4}\right)^{2m-4} \\
&\textrm{(using the inequalities $\binom{N}{r}\leq \left(\frac{eN}{r}\right)^r$ and $2m-4\geq \frac{2e}{3}m$ for $m\geq \log n \geq 22$)}\\
&\leq  \sum_{m=\lceil \log n\rceil}^{\lfloor 2\log n\rfloor } \left(\frac{9em}{256 \log n}\right)^m \left(\frac{256 n \log n}{9m^2}\right)^2  \quad \textrm{(substituting in our bound on $p$)}\\
& \leq  \sum_{m=\lceil \log n\rceil}^{\lfloor 2\log n\rfloor } e^{-2\log n}   \left(\frac{256 n }{9\log n}\right)^2 =  O\left(\frac{1}{\log n}\right),
\end{align*}
where in the last inequality we used the fact that the function
\[x\mapsto x \log \left(\frac{256}{9ex}\right)=x\left(8\log(2) -2\log(3)-1- \log x\right) \] 
is strictly greater than $2$ in the interval $x\in [1,2]$, whence  $\left(\frac{9em}{256 \log n}\right)^{m} 
< e^{-2\log n}$ for all $m$ with $\frac{m}{\log n} \in [1,2]$. It follows 
from Markov's inequality that a.a.s.\ $X=0$ and thus there is no $m$--vertex subset of $V(\Gamma)$ supporting strictly more than $2m-5$ edges of $\Gamma$ for any $m\in [\log n, 2\log n]$, as claimed.
\end{proof}
\begin{proof}[Proof of Theorem~\ref{theorem: lower bound}]
Let $\mathcal{E}$ denote the event that for every $m\in [\log n, 
2\log n]$, every $m$--set of vertices in $\Gamma$ supports at most $2m-5$ edges. By Proposition~\ref{prop: no dense sets}, $\mathcal{E}$ occurs a.a.s., so it is enough to show the conclusions of the theorem hold conditional on $\mathcal{E}$.

Suppose therefore that $\mathcal{E}$ occurs. We shall then prove by induction on $k \geq 0$ that every component of $T_k(\Gamma)$ has support of size at most $\log n$ (which implies both parts of Theorem~\ref{theorem: lower bound}).

\underline{For the base case $k=0$}, suppose that $\Gamma$ contains an
induced thick of order $0$ subgraph on more than $\log n$ vertices.
Let $A\sqcup B$ be any partition of this induced subgraph such
that $\Gamma[A]$ and $\Gamma[B]$ are both non-complete and
$\Gamma[A\sqcup B]=\Gamma[A]*\Gamma[B]$.  Then, provided $n$ is
sufficiently large, there exists $A'\subseteq A$ and $B'\subseteq B$
such that $\Gamma[A']$, $\Gamma[B']$ are both non-complete,
$\Gamma[A'\sqcup B']=\Gamma[A']*\Gamma[B']$ and $\log n\leq \vert
A'\vert +\vert B'\vert \leq 2\log n$.  In other words,
$\Gamma[A'\sqcup B']$ is an induced thick of order $0$ subgraph of
$\Gamma$ on $m$ vertices, for some $m$ with $\log n \leq m \leq 2\log
n$.  By Proposition~\ref{prop: thick of order 0 bound}, we have
$e(\Gamma[A'\sqcup B'])\geq 2m-4$, contradicting our assumption that
$\mathcal{E}$ holds.  Thus $T_0(\Gamma)$ contains no level $0$
component supported on more than $\log n$ vertices, as required.


\underline{For the inductive step}, suppose a component $C$ in $T_{k+1}(\Gamma)$ has support of size at least $\log n$.  Observe that $C$ is obtained by successively gluing together components or cherries $C'_1, C'_2, \ldots $ from $T_{k}(\Gamma)$, each of which has support of size at most $\log n$, and that the sequence of gluing can be done in such a way that for every $i$, the subgraph of $\Gamma$ induced by the union $U_i:=\bigcup_{j=1}^i\mathrm{supp}(C'_j)$ is thick of order at most $k+1$.

 In particular,  we must have $\vert U_{i+1}\vert < \vert U_i\vert +\log n$ for every $i$, and hence there must be a least $i_0$ such that $\vert U_{i_0}\vert >\log n$ satisfying in addition $m:=\vert U_{i_0}\vert <2\log n$. Now the $m$--vertex induced subgraph $\Gamma[U_{i_0}]$ is thick of order at most $k+1$, and hence by Theorem~\ref{theorem: combinatorial characterisation of thickness} must support at least $2m-4$ edges. Since $m\in [\log n, 2\log n]$, this again contradicts our assumption that $\mathcal{E}$ holds.

It follows by induction that, conditional on the a.a.s.\ event
$\mathcal{E}$, for all integers $k \geq 1$, $T_k(\Gamma)$ contains no
component with support of size greater than or equal to $\log n$.  In
particular $\Gamma$ is a.a.s.\ not thick.\end{proof}

Part (ii) of Theorem~\ref{theorem: main} follows directly from
Theorem~\ref{theorem: threshold for thickness of order 2 below that
for order 1} part (i), which is proven below.  (In fact, for $p \geq
(\sqrt{\sqrt{6}-2}+\varepsilon)/\sqrt{n}$, the statement of part (ii)
already follows from Theorem~\ref{theorem: threshold order 1
thickness} part (ii), so we only appeal to Theorem~\ref{theorem:
threshold for thickness of order 2 below that for order 1} to remove
the $\varepsilon$.)

\subsection{Thickness of order two: proof of Theorem~\ref{theorem:
threshold for thickness of order 2 below that for order
1}}\label{section: order 2 thickness} 

To prove Theorem~\ref{theorem:
threshold for thickness of order 2 below that for order 1}, we 
will employ a twist on the 
argument used in~\cite{behrstock2022square} to prove
Theorem~\ref{theorem: threshold order 1 thickness}. 
The crux of the proof of Theorem~\ref{theorem:
threshold order 1 thickness} lay in the analysis of an exploration
process for the square graph $T_1(\Gamma)$ and its comparison with a
supercritical Bienaym\'e--Galton--Watson branching process. We  
describe that exploration process below and explain how a slight modification of it allows us to explore thick of order $2$ rather than thick of order $1$ components, and to keep the associated branching process supercritical for $p$ a little below the threshold for thickness of order $1$.

The exploration process introduced in~\cite[Section 6.1]{behrstock2022square} is as follows. Starting at time $t=0$ from an induced square of $\Gamma\sim \mathcal{G}_{n,p}$ on $\{v_1,v_2,v_3,v_4\}$ with non-edges $v_1v_3$ and $v_2v_4$ one defines a set of discovered vertices $D_0=\{v_1, v_2, v_3, v_4\}$, an (ordered) set of active pairs $A_0=\{v_1v_3, v_2v_4\}$ and a set of reached pairs $R_0=\emptyset$. At every time $t\geq 0$, $D_t$ is a subsets of $V(\Gamma)$, while $A_t$ and $R_t$ are disjoint subsets of $E(\overline{\Gamma})\cap (D_t)^{(2)}$, i.e.\ $A_t$ and $R_t$ are disjoint subsets of non-edges in the subgraph of $\Gamma$ induced by the set of discovered vertices $D_t$.

At each time step $t\geq 0$ of our exploration process, we proceed as follows: 
\begin{enumerate}[1.]
	\item \textbf{If $\vert R_t\vert +\vert A_t\vert$ is large}, meaning 
	$\vert R_t\vert +\vert A_t\vert > (\log n)^4$, then we terminate the process and output $\mathrm{LARGE\ STOP}$.
	\item \textbf{If 
	there are no active pairs left} (i.e.\ if 
	$A_t=\emptyset$), then we terminate the process and output $\mathrm{EXTINCTION\ STOP}$.
	\item \textbf{Otherwise}, we select the first active pair
	$x_1y_1\in A_t$, which by construction induces a $C_4$ in $\Gamma$
	with some pair $F_t\in A_t\cup R_t$.  For every undiscovered
	vertex $z\in V(\Gamma)\setminus D_t$, we test whether or not $z$
	sends an edge of $\Gamma$ to both of $x_1$ and $y_1$; in this way
	we form a set $Z_t:=\{z\in V(\Gamma): \ \{x_1z, y_1z\} \subset
	E(\Gamma)\}$.

	Finally we update our triple $(D_t, A_t, R_t)$ by setting $D_{t+1}=D_t\cup Z_t$, $A_{t+1}= \left(A_t\setminus \{x_1y_1\}\right)\cup \left( (F_t\cup Z_t)^{(2)}\setminus \left({F_t}^{(2)}\cup E(\Gamma)\right)\right)$ and $R_{t+1}=R_t\cup \{x_1y_1\}$.
\end{enumerate}
Note that our update rule under 3.\ above ensures $A_t\cup R_t$ is a collection of non-edges of $\Gamma$ which lie in the same component of the square graph $T_1(\Gamma)$ and in particular preserves the property that every pair in $A_t$ induces a square with some pair $F_t\in D_t^{(2)}$.

The key result in~\cite[Section 6.2]{behrstock2022square} is that for any $\varepsilon>0$ fixed and any $\lambda=\lambda(n)$ satisfying $\sqrt{\sqrt{6}-2}+\varepsilon\leq \lambda \leq 5\sqrt{\log n}$, for $p= \lambda/\sqrt{n}$ the above exploration process is supercritical, and that a.a.s.\ a strictly positive proportion of induced squares in $\Gamma$ are part of large square components, in the sense that their non-edges are part of components of size at least $(\log n)^4$ in $T_1(\Gamma)$; this is the content of~\cite[Lemma 6.2]{behrstock2022square}.

Supercriticality for the process follows from the fact that as long as the number of discovered vertices $D_t$ is small, the  number $\vert A_{t+1}\setminus A_t\vert$ of active pairs discovered at each time step of the process (i.e.\ the offspring distribution) stochastically dominates a random variable $X'\sim \binom{Z'+2}{2}-1$, where $Z'\sim \mathrm{Binom}(n-o(n), p^2 -o(p^2))$. For $p=\lambda/\sqrt{n}$ and $\lambda=\theta(1)$, we have $\mathbb{E}X'= \frac{\lambda^4}{2}+2\lambda^2 +o(1)$, which for fixed $\lambda>\sqrt{\sqrt{6}-2}$ is   at least $1+\eta$ for some fixed $\eta>0$. For such $\lambda$, we can thus apply standard results from branching process theory to conclude that the extinction probability is at most $1-\delta$, for some fixed $\delta>0$, i.e.\ that with probability bounded away from zero our exploration process ends with a $\mathrm{LARGE\ STOP}$.

Having shown that a.a.s.\ a strictly positive proportion of squares
lie in large square components, it easily followed that a.a.s.\ a
positive proportion of non-edges of $\Gamma$ belong to large
square-components of $T_1(\Gamma)$ (see \cite[Corollary
6.3]{behrstock2022square}).  With this fact in hand, a.a.s.\ thickness
of order $1$ for $\Gamma$ was proved in~\cite[Section
6.3-6.5]{behrstock2022square} via a somewhat elaborate vertex
sprinkling argument relying on Janson's inequality and partition
arguments.  That part of the proof, however, only required
$p(n)=\Omega(1/\sqrt{n})$ together with the aforementioned fact that
a.a.s.\ $\Omega(n^2)$ non-edges of $\Gamma$ belong to square
components of order at least $(\log n)^4$ in $T_1(\Gamma)$.

To prove Theorem~\ref{theorem: threshold for thickness of order 2
below that for order 1}, it thus suffices to show that there exists
some absolute constant $c>0$ such that for $\sqrt{\sqrt{6}-2}-c\leq
\lambda \leq 5\sqrt{\log n}$, a.a.s.\ a strictly positive proportion
of non-edges of $\Gamma$ lie in thick of order $2$ components of size
at least $(\log n)^4$.  As we observe below, this can be done with a
modification of the exploration algorithm that increases the expected
number of offspring and thus allows the exploration process to remain
supercritical a little below $p=\sqrt{\sqrt{6}-2}/\sqrt{n}$.  As the
technical modification of the analysis from~\cite{behrstock2022square}
is fairly straightforward, and yields an upper bound on the threshold for thickness of order $2$ which we do not believe is optimal, 
we only sketch the argument and leave the details to the reader.

Given an active pair $x_1y_1 \in A_t$, observe that if there is a
quadruple of vertices $\{x_2, x_3, y_2, y_3\}\subseteq D_t$ such that
setting $X=\{x_1,x_2, x_3\}$ and $Y=\{y_1, y_2, y_3\}$ we have that
the subgraph of $\Gamma$ induced by $X\sqcup Y$ is the complete
bipartite graph on $X\sqcup Y$ with the edge $x_1y_1$ removed, then
the pairs $x_ix_j$ and $y_iy_j$ all lie in the same component of
$T_2(\Gamma)$ as $x_1y_1$ (since $\Gamma[X\sqcup Y]$ is thick of order
$1$ and has $x_1y_1$ as a member of its latch-set but not as the
diagonal of an induced square), and may thus be added to $A_{t+1}$ if
we are looking to explore $T_2(\Gamma)$ rather than $T_1(\Gamma)$.
This leads us to make the following modification of Step 3:

\begin{enumerate}[$3'$]
	\item \textbf{Otherwise}, we select the first active pair $x_1y_1\in A_t$, which by construction induces a $C_4$ in $\Gamma$ with some pair $F_t\in A_t\cup R_t$.   For every undiscovered vertex $z\in V(\Gamma)\setminus D_t$, we test whether or not $z$ sends an edge of $\Gamma$  to both  $x_1$ and $y_1$; in this way we form a set $Z_t:=\{z\in V(\Gamma): \  \{x_1z, y_1z\} \subset E(\Gamma)\}$.

	Next, for every pair of pairs $\{x_2x_3, y_2y_3\}$ drawn from  
	$V(\Gamma)\setminus\left( D_t\cup Z_t\right)$, we test whether we have $x_iy_j \in E(\Gamma)$ for all $(i,j)\neq (1,1)$ and $x_2x_3, y_2y_3'\notin E(\Gamma)$ both holding; if this is the case, we say $x_2x_3$ and $y_2y_3$ are bridge pairs, and we denote by $B'_t$ the collection of all such bridge pairs.

	Finally we update our triple $(D_t, A_t, R_t)$ by setting
	$D_{t+1}=D_t\cup Z_t$, $A_{t+1}= \left(A_t\setminus
	\{x_1y_1\}\right)\cup \left( (F_t\cup Z_t)^{(2)}\setminus
\left({F_t}^{(2)}\cup E(\Gamma)\right)\right) \cup B'_t$ and $R_{t+1}=R_t\cup \{x_1y_1\}$.
\end{enumerate}	

The arguments of~\cite{behrstock2022square} are readily adapted to
show that $\mathbb{E}B'_t\geq
\frac{1}{2}\left(\binom{n-o(n)}{2}\right)^2 p^8(1-p)^6 $ and that as a
result the expected number of offspring in our modified exploration
process when $p=\lambda/\sqrt{n}$ and $\lambda=\theta(1)$ is
$\frac{\lambda^4}{2}+2\lambda^2 + \frac{\lambda^8}{8}+o(1)$.  In
particular, there exist constants $c, \eta>0$ such that for
$\sqrt{\sqrt{6}-2}-c\leq \lambda \leq 2 \sqrt{\sqrt{6}-2}$ the
expected number of offspring in our modified process is at least
$1+\eta$, whereupon the rest of the machinery
from~\cite{behrstock2022square} can be deployed essentially without
any further alterations to ensure the a.a.s.\ existence of a strictly
positive proportion of non-edges in large components of $T_2(\Gamma)$,
and, from there, a.a.s.\ thickness of order $2$.
Theorem~\ref{theorem: threshold for thickness of order 2 below that
for order 1} follows.

\section*{Acknowledgements}
The authors would like to thank Svante Janson for a helpful conversation regarding Theorem 1.2, and an anonymous referee for their careful work and for their comments that improved the clarity and correctness of the paper. The initial research that led to this paper was carried out when the
first author visited the second and third author in Ume{\aa} in January
2023 on the occasion of the fourth Midwinter Meeting in Discrete
Probability.  The hospitality of the Ume{\aa} mathematics department
as well as the financial support of the Simons Foundation which made
this visit possible are gratefully acknowledged.  In addition the
research of the first author was supported by the Simons Foundation 
as a Simons Fellow while the research of the second and
third authors is supported by the Swedish Research Council grant VR
2021-03687.

\end{document}